\documentclass[11pt]{article}
\usepackage{latexsym}
\usepackage{amssymb}
\usepackage{amsfonts,amssymb,amsmath,mathtools,amsthm,comment,bbm}
\usepackage{bbold} 
\usepackage{xcolor}
\usepackage[colorlinks=true, pdfstartview=FitV, linkcolor=blue, citecolor=blue, urlcolor=blue]{hyperref}
\usepackage{url, hypcap}
\usepackage{faktor}

\hypersetup{colorlinks=true, citecolor=darkblue, linkcolor=darkblue}

\definecolor{darkblue}{rgb}{0.0,0,0.7} 

\def\GM{\operatorname{GM}}
\def\RLM{\operatorname{RLM}}
\newcommand{\SP}{\operatorname{SP}}
\newcommand{\CS}{\operatorname{CS}}
\newcommand{\Pl}{\mathsf{Pl}}

\newcommand{\Res}{\operatorname{Res}}

\usepackage{tikz-cd}
\usepackage{tikz}
\usetikzlibrary{arrows,automata}
\usetikzlibrary{patterns,snakes}
\usetikzlibrary{shapes.geometric,calc}
\tikzstyle{arrow}=[thick, ->, >=stealth]
\tikzset{ edge/.style={->,> = latex'} }
\usetikzlibrary{calc,through,backgrounds,shapes,matrix}
\usepackage{graphicx}

\usepackage[colorinlistoftodos]{todonotes}

\newtheorem{thm}{Theorem}[section]
\newtheorem{cor}[thm]{Corollary}
\newtheorem{lem}[thm]{Lemma}
\newtheorem{prop}[thm]{Proposition}
\theoremstyle{definition}
\newtheorem{definition}[thm]{Definition}

\newtheorem{remark}[thm]{Remark}

\numberwithin{equation}{section}
\numberwithin{thm}{section}

\title{Decidability of Krohn--Rhodes complexity for all finite semigroups and automata}

\author{Stuart Margolis, John Rhodes, Anne Schilling}

\date{\today
\footnote{MSC classes: Primary 20M07, 20M10, 20M20, Secondary 54H15}
}

\begin{document}
\maketitle

\begin{abstract}
The Krohn--Rhodes Theorem proves that a finite semigroup divides a wreath product of groups and aperiodic semigroups. 
Krohn--Rhodes complexity equals the minimal number of groups that are needed. Determining an algorithm to compute 
complexity has been an open problem for more than 50 years. The main result of this paper proves that it is decidable if a 
semigroup has complexity $k$ for all $k \geqslant 0$. This builds on our previous work for complexity 1. In that paper we 
proved using profinite methods and results on free Burnside semigroups by McCammond and others that the lower bound from a 2012 paper by Henckell, 
Rhodes and Steinberg is precise for complexity 1. In this paper we define an improved version of the lower bound from the 2012 paper and 
prove that it is exact for arbitrary complexity.
\end{abstract}

\section{Introduction}

The Krohn--Rhodes Theorem, originally known as the Prime Decomposition Theorem, is one of the most important theorems of finite 
semigroup theory. It first appeared in print in~\cite{KR.1965} and has since appeared in many monographs devoted to finite semigroups 
and automata. See for example~\cite{Arbib.1968, TilsonXII, Lallement, RS.2009}. The theorem states that every finite semigroup 
$S$ divides, that is, is a quotient of a subsemigroup, of a wreath product of groups and the semigroup $U$ consisting of two right-zeros 
and an identity. One can choose the groups to be simple groups that divide $S$. A semigroup is called prime if whenever it divides the 
wreath product of two semigroups, then it divides one of the factors. The theorem also shows that a semigroup is prime if and only if it 
is a simple group or is a subsemigroup of $U$.

It follows that a semigroup is aperiodic, that is, has only trivial subgroups if and only if it divides a wreath product of copies of $U$. 
One now defines the (Krohn--Rhodes) complexity $Sc$ of the finite semigroup $S$ to be the least number of groups needed in any 
wreath product decomposition of $S$ as a divisor of wreath products of groups and aperiodic semigroups. It is known that there are 
semigroups of arbitrary complexity. Indeed, the full transformation semigroup $T_k$ on $k$ elements has complexity $k-1$. Therefore, 
if $C_{k}$ is the collection of all semigroups of complexity at most $k$, $C_{k}$ is properly contained in $C_{k+1}$ for all $k$.

\subsection{History and known results}
Complexity was first defined and developed in~\cite{KR.1968} and is extensively described in~\cite{Arbib.1968} and~\cite[Chapter 4]{RS.2009}.
Any explicit decomposition of a semigroup into a wreath product of groups and aperiodic semigroups gives an upper bound to complexity. A nice 
example of an upper bound for complexity is the Depth Decomposition Theorem \cite{TilsonXII}. It states that the complexity of a 
semigroup is bounded by its depth, which by definition is the longest chain of $\mathcal{J}$-classes that contain non-trivial groups. 
The proof gives an explicit decomposition of a semigroup that achieves this bound. 

The full transformation monoid $T_k$ on $k$-elements has depth and complexity equal to $k-1$. It is also known that the monoid 
$M_{k}(F)$ of all $k \times k$ matrices over a finite field $F$ has depth equal to its complexity. On the other hand, as we will recall below, 
the complexity of the symmetric inverse semigroup $\operatorname{SIS}(k)$ is $0$ if $k \leqslant 1$ and $1$ if $k>1$. However its depth is $k-1$.

Another important upper bound is defined by morphism chains. A morphism between semigroups $S$ and $T$  is {\em aperiodic} if it is 
one-to-one on subgroups of $S$. The morphism is called $\mathcal{L}'$ if it separates regular $\mathcal{L}$-classes of $S$. 

One of the most important theorems of finite semigroup theorem, the Fundamental Lemma of Complexity, says that if $f \colon S \rightarrow T$ 
is an aperiodic surjective morphism, then $Sc = Tc$. It is also known that if $f \colon S \rightarrow T$ is a surjective $\mathcal{L}'$ morphism 
then $Sc \leqslant 1+Tc$. 

It follows from these results  about aperiodic and $\mathcal{L}'$ morphisms that if we define $S\theta$ to be the least number of $\mathcal{L}'$ 
morphisms in any factorization of  the morphism $S \rightarrow 1$ from $S$ to the trivial semigroup 
into a composition of aperiodic and $\mathcal{L}'$ morphisms, then $\theta$
is an upper bound for complexity. That is, for any semigroup $S$, we have $Sc \leqslant S\theta$. 
It is proved in~\cite[Chapter 9]{Arbib.1968} that if $S$ is a completely regular semigroup, then $Sc = S\theta$. In particular, this proves that 
complexity is decidable for completely regular semigroups. On the other hand, it is not difficult to show that $\operatorname{SIS}(k)\theta = k-1$, so that 
$\theta$ is a proper upper bound. 

Recall that if $(P,T)$ and $(Q,S)$ are transformation semigroups, then the wreath product $(P,T) \wr (Q,S)$ is the transformation semigroup 
$(P \times Q,W)$, where $W$ is the direct product $F(Q,T) \times S$ with $F(Q,T)$ the set of functions from $Q$ to $T$. If 
$(f,s) \in W$ and $(p,q) \in P \times Q$, then we define $(p,q)(f,s)=(p(qf),qs)$ if both coordinates are defined and undefined otherwise. 
For background on wreath products, see~\cite{Eilenberg.1976, Arbib.1968, RS.2009}.

One can check that if $A$ is an aperiodic semigroup and $S$ is an arbitrary semigroup, then the projection $\pi_{A} \colon A \wr S \rightarrow S$ 
is an aperiodic morphism. If $G$ is a group, then the projection $\pi_{G} \colon G \wr S \rightarrow S$ is a $\mathcal{L}'$ morphism.  
It follows easily that if we define $S\widehat{\theta}=\operatorname{min}\{T\theta \mid S$ is a quotient of $T\}$, then $S\widehat{\theta}= Sc$. 
Like the formal definition of complexity, this result requires searching through the infinite collection of all semigroups mapping onto $S$ and 
is thus not a priori computable. 

Early lower bounds~\cite{lowerbounds1, lowerbounds2} were based on the length of certain subsemigroup chains. A subsemigroup $T$ of $S$ 
that is generated by a chain of $\mathcal{L}$-classes of $S$ is called an absolute type I subsemigroup of $S$. The type II subsemigroup, $S_{II}$, 
is the smallest subsemigroup of $S$ containing the idempotents and is closed under weak conjugation: if $sts=s$ in $S$, then $sS_{II}t \cup tS_{II}s$ 
is contained in $S_{II}$. If one defines $Sl$ to be the length of the longest chain of alternating absolute type I and type II subsemigroups that contains 
a non-aperiodic type I subsemigroup whose type II subsemigroup is non-aperiodic, then it is proved in \cite{lowerbounds2} that $Sl \leqslant Sc$.

Note that $\mathcal{L}'$ morphisms lead to an important functor on the category of finite semigroups. Let $S$ be a semigroup. Right multiplication 
of an $\mathcal{L}$-class by an element of a semigroup is either an $\mathcal{L}$-class in the same $\mathcal{J}$-class or is an 
$\mathcal{L}$-class in a strictly smaller $\mathcal{J}$-class in the $\mathcal{J}$-order. Then it can be proved that $S$ acts by partial 
functions on its collection $\operatorname{Reg}(L)$ of its regular $\mathcal{L}$-classes if we define $l\cdot s =ls$ if $ls$ belongs to the 
same $\mathcal{J}$-class and is undefined otherwise. The image $S^{\mathcal{L}'}$ is the minimal $\mathcal{L}'$ image of $S$ and is the 
object part of the aforementioned functor. This is part of the semi-local theory of finite semigroups. See \cite{RS.2009} for details.

If $S$ is any inverse semigroup, it is not difficult to see that $S_{II}$ is the semilattice of idempotents of $S$. Therefore, $Sl \leqslant 1$. Since 
it is known that $Sc \leqslant 1$ if $S$ is inverse we have $Sc = Sl$ in this case. It is proved in~\cite[Chapter 9]{Arbib.1968} that if $S$ is a 
completely regular semigroup, then $Sc =Sl$. Thus for the two classically studied classes of semigroups, complexity was known to be decidable 
at the beginning of the theory in the 1960s.

The two lower bounds above compute complexity for semigroups that divide the wreath product of a group and an aperiodic semigroup in 
either order. More precisely, a semigroup $S$ divides a wreath product $A \wr G$, where $A$ is aperiodic and $G$ is a group if and only if 
$S_{II}$ is aperiodic \cite{lowerbounds2}. A semigroup $S$ divides a wreath product $G \wr A$, where $G$ is a group and $A$ is aperiodic 
if and only if $S^{\mathcal{L}'}$ is an aperiodic semigroup \cite{Rhodes31}. Thus for these two classes of semigroups, complexity is decidable.

Another important class of semigroups in which $Sl=Sc$ is the class of semigroups with 2 non-zero $\mathcal{J}$-classes. It is not hard to 
reduce this assertion to the case of a small monoid, that is a monoid consisting of a group of units and a 0-simple ideal. Since the depth of 
such a monoid is at most 2, the result comes down to proving that if $Sl=1$, then $Sc=1$. This is proved in~\cite{Tilson.1973}. The proof 
introduces the concept of cross-sections. This was developed further by Rhodes who called this the Presentation Lemma~\cite{AHNR.1995, RS.2009}. 
This along with the allied notion of a flow on a semigroup~\cite[Sections 2-3]{HRS.2012} is the central tool in studying complexity.

\subsection{Preliminaries, main result, and proof outline}

For in depth backgrounds on the decomposition and complexity theory of finite semigroups, see~\cite{Arbib.1968} and~\cite[Chapter 4]{RS.2009}. 
For details of all assertions in this subsection, the reader should turn to these references. We give an overview of the tools we need and an outline 
of the proof of the main theorem.

A semigroup $S$ is bi-transitive if it has a faithful transitive action by partial functions on the right of some set $X$ and the left of some set $Y$. 
In the literature~\cite{Arbib.1968, RS.2009}, bi-transitive semigroups are called generalized group mappings abbreviated by $\operatorname{GGM}$. 
We will use this terminology in this paper.

It is known that a semigroup $S$ is $\operatorname{GGM}$ if and only if it has a unique 0-simple ideal $I(S) \approx M^{0}(A,G,B,C)$ such that 
the right (left) Sch\"utzenberger representation of $S$ is faithful on some (equivalently any) $\mathcal{R}$($\mathcal{L})$-class in $I(S)$. 
A $\operatorname{GGM}$ semigroup is called {\em group mapping} denoted by $\operatorname{GM}$ if the group $G$ in $I(S)$ is non-trivial. 

By fixing an $\mathcal{R}$-class $R$ in $I(S)$ we can identify it as a set with $G \times B$. We thus have a faithful action, that is, a transformation 
semigroup $(G \times B,S)$. The set $B$ is an index for the set of $\mathcal{L}$-classes in $I(S)$. The action mentioned in the previous subsection 
on regular $\mathcal{L}$-classes restricts to an action by partial functions on the set $B$. The image of this action is called the {\em right latter mapping} 
image of $S$ and is denoted by $\operatorname{RLM}(S)$.

The importance of $\operatorname{GM}$ semigroups and their $\operatorname{RLM}$ images is summed up by the following results.

\begin{thm}
Let $S$ be a $\operatorname{GM}$ semigroup.
\begin{enumerate}
\item
$(G \times B,S)$ divides $G \wr (B,\operatorname{RLM}(S))$.
\item
$S^{\mathcal{L}'} = \operatorname{RLM}(S)$.
\item
For every semigroup $T$, there is a $\operatorname{GM}$ quotient semigroup $S$ such that $Tc=Sc$.
\end{enumerate}
\end{thm}

It follows that the question of decidability of complexity can be reduced to the case that $S$ is a $\operatorname{GM}$ semigroup and 
whether $Sc=\operatorname{RLM}(S)c$ or $Sc = 1+\operatorname{RLM}(S)c$. The Presentation Lemma~\cite{AHNR.1995} and the related 
notion of Flows on Automata~\cite{HRS.2012} give necessary and sufficient, though not necessarily computable conditions for 
$Sc=\operatorname{RLM}(S)c$. We highlight the associated decomposition theorem. Details about flows will be given in Section~\ref{Section.flow}.

\begin{thm}[The Presentation Lemma Flow Decomposition Theorem]
\label{PLDT}
Let $(G \times B,S)$ be a $\operatorname{GM}$ transformation semigroup and let $k > 0$. Then $Sc \leqslant k$ if and only if 
$\operatorname{RLM}(S)c \leqslant k$ and there is a transformation semigroup $T$ with $Tc \leqslant k-1$ such that 
\[
	(G \times B,S) \prec (G \wr \mathsf{Sym}(B) \wr T) \times (B,\operatorname{RLM}(S)).
\]
\end{thm}

Note that given a $\operatorname{GM}$ semigroup $S$ it is easy to compute $G$, $\mathsf{Sym}(B)$ and $\operatorname{RLM}(S)$. Therefore, 
computing complexity comes down to deciding if the semigroup $T$ in the statement of Theorem~\ref{PLDT} is computable. 
The main theorem of~\cite{MRS.2023} shows this can be done in the case that $Sc = 1$. In this case, we have to find an aperiodic semigroup 
$T$ in Theorem~\ref{PLDT}.

\begin{thm}[Margolis, Rhodes, Schilling 2023]
If $Sc \leqslant 1$, then the aperiodic semigroup $T$ in Theorem~\ref{PLDT} has order bounded by a computable function of the cardinality 
of $S$.
\end{thm}

\begin{thm}[Margolis, Rhodes, Schilling 2023]
\label{compdec}
It is decidable whether a finite semigroup has complexity at most 1.
\end{thm}

We will give a more precise version of Theorem \ref{compdec} later in the paper. We now state the main result of this paper. 

\begin{thm}\label{dec=k}
It is decidable whether a finite semigroup has complexity at most $k$ for all $k \geqslant 0$.
\end{thm}

The proof of decidability of complexity $k$ for all $k \geqslant 0$ follows by induction on $k$. The case $k \leqslant 1$ is given by 
Theorem~\ref{compdec}. Given $k > 1$ and a $\operatorname{GM}$ transformation semigroup $(G \times B,S)$ we effectively find a relational 
morphism $\rho$ onto a semigroup $T$ with $Tc = Sc-1$. We then use Theorem \ref{compdec} to prove that the derived transformation semigroup 
$D(\rho)$ has complexity at most 1. The result now follows from the Derived Semigroup Theorem~\cite{RS.2009}. More details about the proof 
and the proof itself are given in the next sections.

\subsection*{Acknowledgements}
We wish to thank Jean-Camille Birget, Sam van Gool, Pedro Silva, Ben Steinberg and Howard Straubing for helpful comments and discussions.

SM thanks Bar-Ilan University for generous retirement benefits. 

JR thanks the University of California, Berkeley for generous retirement benefits.

AS was partially supported by NSF grant DMS--2053350.

\section{Background}

In this section, we review and extend results from~\cite{AmigoDowling} and~\cite{HRS.2012} which we will need for the proof.

\subsection{The connection between Rhodes lattices and set-partition lattices} 
\label{subsection.SPRhodes}

Let $S$ be a $\operatorname{GM}$ semigroup with $0$-minimal ideal $M^{0}(A,G,B,C)$. We define two lattices of interest and give the 
connections between them. The first is the set-partition lattice $\SP(G \times B)$. This is the lattice whose elements are all pairs $(Y,\Pi)$,
where $Y$ is a subset of $G\times B$ and $\Pi$ is a partition on $Y$. Here $(Y,\Pi) \leqslant (Z,\Theta)$ if $Y \subseteq Z$ and 
the $\Pi$-class of $y$ is contained in the $\Theta$-class of $y$ for all $y\in Y$.

The second lattice is the Rhodes lattice $Rh_{B}(G)$. We review the basics. For more details see \cite{AmigoDowling}. Let $G$ be a finite group 
and $B$ a finite set. A partial partition on $B$ is a partition $\Pi$ on a subset $I$ of $B$. We also consider the collection of all functions
$F(B,G)$, $f \colon I \rightarrow G$ from subsets $I$ of $B$ to $G$. The group $G$ acts on the left of $F(B,G)$ by $(gf)(b) = gf(b)$ for 
$f \in F(B,G), g \in G, b \in \operatorname{Dom}(f)$. An element $Gf = \{gf \mid f \colon I \rightarrow G, I \subseteq X, g \in G\}$ of the quotient set $F(B,G)/G$ is 
called a cross-section with domain $I$. An SPC (Subset, Partition, Cross-section) over $G$ is a triple $(I,\Pi,f)$, where $I$ is a subset of $B$, $\Pi$ is a partition 
of $I$ and $f$ is a collection of cross-sections one for each $\Pi$-class $\pi_{i}$ with domain $\pi_{i}$. 

If the classes of $\Pi$ are $\{\pi_{1},\pi_{2}, \ldots, \pi_{k}\}$, 
then we sometimes write $\{(\pi_{1},f_{1}), \ldots, (\pi_{k},f_{k})\}$, where $f_{i}$ is the cross-section associated to $\pi_{i}$. For brevity we denote this 
set of cross-sections by $[f]_{\Pi}$.  We let $Rh_{B}(G)$ denote the set of all SPCs on $B$ over the group $G$ union a new element 
$\Longrightarrow\Longleftarrow$  that we call {\em contradiction} and is the top element of the lattice structure on $Rh_{B}(G)$.  Contradiction 
occurs  because the join of two SPCs need not exist. In this case we say that the contradiction is their join.

The partial order on $Rh_{B}(G)$ is defined as follows.

We have $(I,\Pi,[f]_{\Pi}) \leqslant (J,\tau,[h]_{\tau})$ if:
\begin{enumerate}
\item
$I \subseteq J$;
\item
Every block of $\Pi$ is contained in a (necessarily unique) block of $\tau$;
\item
If the $\Pi$-class $\pi_{i}$ is a subset of the $\tau$-class $\tau_{j}$, then the restriction of $h$ to $\pi_{i}$ equals $f$ restricted to $\pi_{i}$ as 
elements of $F(B,G)/G$. That is, $[h|_{\pi_{i}}] = [f|_{\pi_{i}}]$.
\end{enumerate}

See \cite[Section 3]{AmigoDowling} for the definition of the lattice structure on $Rh_{B}(G)$. The same set is also the set underlying the Dowling lattices, 
which has a different partial order. For the connection between Rhodes lattices and Dowling lattices see~\cite{AmigoDowling}.

We note that $\SP(G \times B)$ is isomorphic to the Rhodes lattice $Rh_{G\times B}(1)$ of the trivial group over the set $G \times B$. We need only 
note that a cross-section to the trivial group is a partial constant function to the identity and can be omitted, leaving us with a set-partition pair. There 
are no contradictions for Rhodes lattices over the trivial group, and in this case the top element is the pair $(G \times B,(G \times B)^{2})$. Despite 
this, we prefer to use the notation $\SP(G \times B)$ instead of $Rh_{G\times B}(1)$.

Conversely, we can find a copy of the meet semilattice of $Rh_{B}(G)$ as a meet subsemilattice of $\SP(G\times B)$. We begin with the following 
important definition.

\begin{definition}
A subset $X$ of $G \times B$ is a cross-section if whenever $(g,b),(h,b) \in X$ then g = h. That is, $X$ defines a cross-section of the projection 
$\theta \colon X \rightarrow B$. Equivalently, $X^{\rho} \subseteq B \times G$, the reverse of $X$, is the graph of a partial function 
$f_{X} \colon B \rightarrow G$.  An  element $(Y,\Pi) \in \SP(G \times B)$ is a cross-section if  every partition class $\pi$ of $\Pi$ is a cross-section.  
\end{definition}

An element $(Y,\Pi)$ that is not a cross-section is called a {\em contradiction}. That is, $(Y,\Pi)$ is a contradiction if some $\Pi$-class $\pi$ contains 
two elements $(g,b),(h,b)$ with $g \neq h$. We note that the set of cross-sections is a meet subsemilattice of $\SP(G \times B)$ and the set of 
contradictions is a join subsemilattice of $\SP(G \times B)$.

We will identify $B$ with the subset $\{(1,b) \mid b \in B\}$ of $G \times B$. Then  $G\times B$ is the free left $G$-act on 
$B$ under the action $g(h,b)=(gh,b)$. This action extends to subsets and partitions of $G\times B$ . Thus $\SP(G\times B)$ is a left $G$-act. 
An element $(Y,\Pi)$ is {\em invariant} if $G(Y,\Pi) = (Y,\Pi)$. It is easy to see that $(Y,\Pi)$ is invariant if and only if:
\begin{enumerate}
\item
$Y=G \times B'$ for some subset $B'$ of $B$.
\item
For each $\Pi$-class $\pi$, $G\pi \subseteq \Pi$ and is a partition of $G\times B''$ where $B''\subseteq B'$.
\end{enumerate}

Thus $(Y,\Pi)$ is invariant if and only if $Y=G\times B'$ for some subset $B'$ of $B$ and there is a partition $B_{1},\ldots, B_{n}$ of $B'$ such 
that for all $\Pi$-classes $\pi$, $G\pi$ is a partition of $G\times B_{i}$ for a unique $1 \leqslant i \leqslant n$. 

Let $\CS(G\times B)$ be the set of invariant cross-sections $(Y,\Pi)$ in $\SP(G \times B)$. Then $\CS(G\times B)$ is a meet subsemilattice of 
$\SP(G\times B)$. We claim it is isomorphic to the meet semilattice of $Rh_{B}(G)$. Indeed, let $(G\times B',\Pi)$ be an invariant cross-section. 
Using the notation above, pick  $\Pi$-classes, $\pi_{1}, \ldots , \pi_{n}$ such that $G\pi_{i}$ is a partition of $G\times B_{i}$ for $i=1, \ldots, n$. 
Since $\pi_{i}$ is a cross-section its reverse is the graph of a partial function $f_{i} \colon B_{i}\rightarrow G$. We map 
$(G\times B',\Pi)$ to $(B',\{B_{1},\ldots , B_{n}\},[f]_{\{B_{1},\ldots , B_{n}\}}) \in Rh_{B}(G)$, where the component of $f$ on $B_{i}$ is $f_i$. 
It is clear that this does not depend on the representatives $\pi$ and we have a well-defined function $F \colon \CS(G \times B) \rightarrow Rh_{B}(G)$. 
It is easy to see that $F$ is a morphism of meet semilattices.

Conversely, let $(B',\Theta,[f]) \in Rh_{B}(G)$. We map $(B',\Theta,[f])$ to $(G \times B',\Pi) \in \CS(G \times B)$, where $\Pi$ is defined as 
follows. If $\theta$ is a $\Theta$-class, let $\widehat{\theta} =\{(bf,b) \mid b \in \theta\}$. Let $\Pi$ be the collection of all subsets of the 
form $g\widehat{\theta}$, where $\theta$ is a $\Theta$-class and $g \in G$. The proof that $(G \times B',\Pi) \in \CS(G \times B)$ and 
that this assignment is the inverse to $F$ above and gives an isomorphism between the meet semilattice of $Rh_{B}(G)$ and $\CS(G \times B)$ 
is straightforward and left to the reader. Furthermore, if the join of two SPC in $Rh_{B}(G)$ is an SPC (that is, it is not the contradiction), then this 
assignment preserves joins. Therefore we can identify the lattice $Rh_{B}(G)$ with $\CS(G \times B)$ with a new element $\Longrightarrow\Longleftarrow$ 
added when we define the join of two elements of $\CS(G \times B)$ to be $\Longrightarrow\Longleftarrow$ if their join in $\SP(G\times B)$ is a 
contradiction. We use $\CS(G\times B)$ for this lattice as well.

We record this discussion in the following proposition.

\begin{prop}\label{Updown}
The set-partition lattice $\SP(G\times B)$ is isomorphic to the Rhodes lattice $Rh_{G\times B}(1)$. The Rhodes lattice $Rh_{B}(G)$ is isomorphic 
to the lattice $\CS(G\times B)$.
\end{prop}

\subsection{The monoid of closure operators on a lattice }

We make extensive use of tools from \cite{HRS.2012}. In this subsection we recall the definition of the monoid of closure operators on 
the direct product of a lattice with itself and some unary operators on this monoid.

Let $L$ be a lattice and let $L^{2} = L \times L$. Let $\Phi$ be a closure operator on $L^2$. Recall this means that $\Phi$ is an order preserving, 
extensive (that is, for all $l \in L, l \leqslant l\Phi$), idempotent function on $L^2$. A {\em stable pair} for $\Phi$ is a closed element of $\Phi$. Thus 
a stable pair $(l,l') \in L^2$ is an element such that $(l,l')\Phi = (l,l')$.  The stable pairs of $\Phi$ are a meet closed subset of $L \times L$. Conversely, 
each meet closed subset of $L \times L$ is the set of stable pairs for a unique closure operator on $L \times L$. So for us, $\Phi$ is a binary relation 
on $L$. The Boolean matrix associated to $\Phi$ is of dimension $|L| \times |L|$. It has a 1 in position $(l_{1},l_{2})$ if $(l_{1},l_{2})$ is a stable pair and 
a 0 otherwise.

With this identification, the collection $\mathcal{C}(L^{2})$ of all closure operators on $L^2$ is a submonoid of the monoid $B(L)$ of binary relations on 
$L$~\cite[Proposition 2.5]{HRS.2012}. As is well known, $B(L)$ is isomorphic to the monoid $ M_{n}(\mathcal{B})$ of $n \times n$ matrices over the 
2-element Boolean algebra $\mathcal{B}$, where $n=|L|$. We thus also consider $\mathcal{C}(L^{2})$ to be a monoid of $n \times n$ Boolean matrices.
Let $L$ be either $Rh_{B}(G)$ or $\SP(G\times B)$.

We use all the definitions of operators on the monoid  $\mathcal{C}(L^{2})$ of closure operators on  lattice $L^{2} = L \times L$, where $L$ is a 
lattice.  

We now recall the definition of some important unary operations on $\mathcal{C}(L^{2})$. Let $f \in \mathcal{C}(L^{2})$. Define 
$\operatorname{Dom}(f)=\{l \mid \exists l', (l,l') \in f\}$. The operator $\overleftarrow{f}=\{(l,l)\mid l \in \operatorname{Dom}(f)\}$ is called the 
back-flow of $f$ and is also a member of $\mathcal{C}(L^{2})$. The Kleene star of $f$, denoted $f^{\star}$, is the intersection of $f$ with the 
diagonal $\{(l,l) \mid l \in L\}$. Note that $f^{\star}$ is also in $\mathcal{C}(L^{2})$. The operator $f^{\omega+{\star}}$ is defined to be $f^{\omega}f^{*}$, 
where $f^{\omega}$ is the unique idempotent in the subsemigroup generated by $f$. Thus $(l,l') \in f^{\omega + \star}$ if and only if $(l,l') \in f^{\omega}$ 
and $(l',l') \in f$. See~\cite[Section 2]{HRS.2012} for motivation for the names of these operators and their basic properties. 

\subsection{Flows and the Presentation Lemma Flow Decomposition Theorem} 
\label{Section.flow}

Let $X$ be a generating set for our fixed $\operatorname{GM}$ semigroup $S$. We review the definition of flows from  an automaton with alphabet 
$X$  to the set-partition $\SP(G\times B)$ and to the Rhodes lattice $Rh_{B}(G)$. For more details, see \cite[Sections 2-3]{HRS.2012}.

Let $(G \times B,S)$ be $\operatorname{GM}$ and $X$ a generating set for $S$. By deterministic automaton we mean an automaton such 
that each letter defines a partial function on the state set.

\begin{definition}
Let $\mathcal{A}$ be a deterministic automaton with state set $Q$ and alphabet $X$. 
A {\em flow} to the lattice $\SP(G\times B)$ on $\mathcal{A}$ is a function $f \colon Q \rightarrow \SP(G \times B)$ such that for each $q \in Q, x \in X$, 
with $qf =(Y,\Pi)$ and $(qx)f = (Z,\Theta)$ we have:
\begin{enumerate}
\item
$Yx \subseteq Z$.
\item
Multiplication by $x$ considered as an element of $S$ induces a partial 1-1 map from $Y/\Pi$ to $Z/\Theta$.

That is, for all $y,y' \in Y$ we have $y,y'$ are in a $\Pi$-class if and only if $yx,y'x$ are in a $\Theta$-class whenever $yx,y'x$ are both defined.
\item
{\bf The Cross Section Condition:} For all $q \in Q$, $qf$ is a cross-section.
\end{enumerate}
\end{definition}

We use Proposition \ref{Updown} to give a definition of a flow to the Rhodes lattice $Rh_{B}(G)$.

\begin{definition}
Let $\mathcal{A}$ be a deterministic automaton with state set $Q$. A {\em flow} to the lattice $Rh_{B}(G)$ is a flow to $\SP(G\times B)$ such 
that for each state $q$, $qf \in \CS(G\times B)$. That is, $qf$ is a $G$-invariant cross-section.  
\end{definition}

\begin{remark}
It follows from the original statement of the Presentation Lemma \cite{AHNR.1995} and from \cite{Redux} that if $S$ has a flow with respect to some 
automaton over $\SP(G\times B)$, then it has a flow from the same automaton over $Rh_{B}(G)$. This is because the minimal injective congruence as 
discussed in \cite{Redux} is $G$-invariant. The proof of the equivalence of the Presentation Lemma and Flows
in~\cite[Section 3]{HRS.2012} works as well for the Presentation Lemma in the sense of~\cite{AHNR.1995}. We will always assume that flows are 
over the Rhodes lattice $Rh_{B}(G)$.    
\end{remark}

 If $\mathcal{A}$ is an automaton with alphabet $X$ and state set $Q$, recall that its completion is the automaton $\mathcal{A}^{\square}$ that 
 adds a sink state $\square$ to $Q$ and declares that $qx = \square$ if $qx$ is not defined in $\mathcal{A}$. A flow on $\mathcal{A}$ is a complete 
 flow if on $\mathcal{A}^{\square}$ extending $f$ by letting $\square f = (\emptyset, \emptyset)$ the bottom of both the set-partition and Rhodes 
 lattices remains a flow and furthermore, for all $(g,b) \in G \times B$, there is a $q\in Q$ such that $(g,b) \in Y$, where $qf=(Y,\Pi)$. All flows in 
 this paper will be complete flows.

 Flows are related to the Presentation Lemma \cite{AHNR.1995}, \cite[Section 4.14]{RS.2009}. They give a necessary and sufficient condition for 
 $Sc =\operatorname{RLM}(S)c$, where $S$ is a $\operatorname{GM}$ semigroup.

\begin{thm}[The Presentation Lemma-Flow Version]
\label{PLflow}
Let $(G \times B,S)$ be a $\operatorname{GM}$ transformation semigroup with $S$ generated by $X$. Let $k> 0$ and assume that 
$\operatorname{RLM}(S)c = k$.  Then $Sc = k$ if and only if:
\begin{itemize}
	\item There is an $X$ automaton $\mathcal{A}$ whose transformation semigroup $T$ has complexity strictly less than $k$ and a complete flow
	$f \colon Q \rightarrow Rh_{B}(G)$.
	\item In this case, $(G \times B,S)$ divides $(G \wr \mathsf{Sym}(B) \wr T) \times (B,\operatorname{RLM}(S))$.
\end{itemize}
\end{thm}

We emphasize the decomposition in Theorem \ref{PLflow} in the following Corollary.

\begin{cor}[{The Presentation Lemma Flow Decomposition Theorem}]
\label{PLdecomp}
 Let $(G \times B,S)$ be $\operatorname{GM}$ and let $k > 0$.  Then $Sc \leqslant k$ if and only if $\operatorname{RLM}(S)c \leqslant k$ and 
 there is a transformation semigroup $T$ with $Tc \leqslant k-1$ such that 
$$ (G \times B,S) \prec (G \wr \mathsf{Sym}(B) \wr T) \times (B,\operatorname{RLM}(S)).$$
\end{cor}

We call the semigroup $\Res = (G \wr \mathsf{Sym}(B) \wr T)$ in Corollary \ref{PLdecomp} {\em the resolution semigroup} of the flow. 
The minimal ideal $K(\Res)$ resolves the $0$ of $S$ in a sense akin to algebraic geometry. 

For example, if $S$ is a $\operatorname{GM}$ inverse semigroup, we can take $T$ to be the trivial semigroup. Then $\operatorname{RLM}(S)$ 
is the fundamental image of $S$ and the decomposition $S$ divides $(G \wr \mathsf{Sym}(B)) \times \operatorname{RLM}(S)$ follows from a
theorem of McAlister and Reilly~\cite{McAlReilly} from 1977.  In this case, the graph of the relational morphism from $S$ to $G \wr \mathsf{Sym}(B)$ 
is an $E$-unitary cover of $S$ and the group $G \wr \mathsf{Sym}(B)$ resolves the problem of the maximal group image of $S$ being the trivial group, 
since we assume $\operatorname{GM}$ semigroups have a zero. 

The Presentation Lemma was first used in 1973 \cite{Tilson.1973}. It was used from the early 1970s by Rhodes, but only first appeared in 
print in~\cite{AHNR.1995}.

We give an equivalent version of flows that will be important in our revised version of the evaluation transformation semigroup. The discussion here 
follows~\cite[Section 4.14]{RS.2009}. We  recall the definition of the derived transformation semigroup of a relational morphism of transformation 
semigroups as defined in~\cite[Definition 4.14.2]{RS.2009} and \cite[Chapter 3]{Eilenberg.1976}. See these references for more details and for the proof 
of Theorem \ref{DST} below.

Let $\varphi \colon (P,S) \rightarrow (Q,T)$ be a relational morphism of 
transformation semigroups. For each $s \in S$ we pick an element 
$\overline{s} \in T$ that covers $s$. This is called a parameterization of 
the relation. The state set of the derived transformation semigroup 
$D(\varphi)$ is the graph $\#r \subseteq P \times Q$ of $r$. A 
transformation of $D(\varphi)$ is a triple $(q,s,q\overline{s})$. The action is 
given by $(p,q)(q,s,q\overline{s})=(ps,q\overline{s})$ if $ps$ and 
$q\overline{s}$ are defined and undefined otherwise. Informally, the
transformations of $D(\varphi)$ are the restrictions of the actions of $S$
to the sets $q\varphi^{-1}$ relative to the parameterization. The 
importance of the derived transformation semigroup is the following version
of the Derived Semigroup Theorem.

\begin{thm}\label{DST}
Let $\varphi \colon (P,S) \rightarrow (Q,T)$ be a relational morphism of transformation semigroups. Then $(P,S)$ divides $D(\varphi) \wr (Q,T)$. 
Furthermore, if $(P,S)$ divides $X \wr (Q,T)$ for some transformation semigroup $X$, then $D(\varphi)$ divides $X$.
\end{thm}

The following is an equivalent formulation of the Presentation Lemma \cite[Theorem 4.14.19]{RS.2009}. See~\cite[Section 3]{HRS.2012} for a proof of 
equivalence. By a congruence on a transformation semigroup $(P,S)$ we mean an equivalence relation $\equiv$ on $P$ such that for all $p,q \in P$ 
and whenever both $ps$ and $qs$ are defined, for $s \in S$, then $ps \equiv qs$. The quotient transformation semigroup is the faithful image of the 
action of $S$ on $\faktor{P}{\equiv}$ defined by $[p]_{\equiv}s=[p's]_{\equiv}$ if there is a $p\equiv p'$ such that $p's$ is defined. We warn the reader 
that the semigroup of the quotient may not be a quotient of $S$. A congruence $\equiv$ on a transformation semigroup is injective if the quotient 
transformation semigroup is a semigroup of partial 1-1 maps. Equivalently, if for all $p,q \in P$ such that both $ps,qs$ are defined for $s \in S$, we 
have $ps \equiv qs$ if and only if $p\equiv q$. 

Let $(G \times B,S)$ be a $\GM$ transformation semigroup and $\varphi \colon (G \times B,S) \rightarrow (Q,T)$ a relational morphism of 
transformation semigroups. A congruence $\equiv$  on $D(\varphi)$ is admissible if whenever $((g,b),q) \equiv ((g',b'),q')$ then $q=q'$. An 
admissible congruence $\equiv$ is a cross section if whenever $((g,b),q) \equiv ((h,b),q)$, then $g = h$.

\begin{thm}
Let $(G \times B,S)$ be a $\GM$ transformation semigroup and let $(Q,T)$ be a transformation semigroup. Then there is a flow 
$F \colon Q \rightarrow Rh_{B}(G)$ if and only if there is a relational morphism of transformation semigroups 
$\varphi \colon (G \times B,S) \rightarrow (Q,T)$ such that the minimal injective congruence on $D(\varphi)$ is a cross section. In this case, 
$(G \times B,S)$ divides $(G \wr Sym(B) \wr T) \times \RLM(S)$.
\end{thm}

More generally, let $(G \times B,S)$ be a $\GM$ transformation semigroup and let $(Q,T)$ be a transformation semigroup and 
$\varphi \colon (G \times B,S) \rightarrow (Q,T)$. We define the subgroup $N(\varphi)$ to be the maximal subgroup of $G$ that is contained in a partition 
class of the minimal injective congruence on $D(\varphi)$. Thus $\varphi$ leads to a flow if and only if $N(\varphi)$ is the trivial subgroup. 
Note that $N(\varphi)$ is the obstruction to $\varphi$ giving a flow.

We use this concept to describe the maximal subgroup of $G$ that is pointlike with respect to the pseudovariety $C_k$ of all semigroups with complexity 
at most $k$. We begin by reviewing the definition of normal subgroup spreads as defined in~\cite[Notation 3.1a,b]{Rhodes.1977}. A normal subgroup spread 
of a semigroup $S$ is an assignment $N_e$ of a normal subgroup to each maximal subgroup $G_e$ of $S$ for each idempotent $e \in S$ such
that if $e \mathcal{J} f$, then $N_e$ is isomorphic to $N_f$ via any Rees-Green isomorphism of $G_e$ onto $G_f$. In other words, if
for two idempotents $e^2=e \mathcal{J} f^2=f$ and $xe_ay=e_b$ for some $x,y \in S^{\mathbb{1}}$, then the isomorphism from $G_e$ onto $G_f$ defined by
$g \mapsto x g y$ for $g \in G_e$  maps $N_e$ onto $N_f$. We denote by $N_J$ any $N_e$ contained in the $\mathcal{J}$-class
$J$ and we denote by $\mathcal{N} = \{N_J \mid \text{$J$ is a regular $\mathcal{J}$-class of $S$}\}$ a normal subgroup spread of $S$.

Furthermore, define
\[
	\mathsf{Reg}\mathcal{J}(S) = \{ J \mid \text{$J$ is a regular $\mathcal{J}$-class of $S$}\}.
\]
Let $J$ be a regular $\mathcal{J}$-class of a semigroup $S$. We define $\operatorname{GM}_{J}(G/N)$ to be the semigroup denoted by 
$\operatorname{GM}(J,G,N)$ in~\cite[Chapter 8]{Arbib.1968}. In brief, we first define $S/N$ to be the quotient semigroup of $S$ defined by the unique 
congruence that identifies all elements in the ideal $I(J)=\{s \in S \mid J \cap S^{1}sS^{1} =\emptyset\}$ and whose restriction to $G$ are the cosets of 
$N$ and is the identity on $S - (I(J) \cup J)$. Then $\operatorname{GM}_{J}(G/N)$ is the quotient of $S/N$ by the $\operatorname{GM}$ congruence 
associated to the 0-minimal $\mathcal{J}$-class of $S/N$. We recall the convention that if $G=N$, then $\operatorname{GM}_{J}(G,G)$ is the trivial semigroup.

If $S$ is a $\GM$ semigroup with distinguished $\mathcal{J}$-class with $I(S) \approx M^{0}(A,G,B,C)$, then there is a surjective morphism from 
$S$ to $\GM_{J}(G,N)$ that induces a surjective morphism from $I(S)$ to the distinguished ideal $M^{0}(A_{N},G_{N},B_{N},C_{N})$ of $\GM_{J}(G,N)$. 
This also gives a surjective morphism of transformation semigroups on the corresponding $\GM$ transformation semigroups. 

We note that if $\varphi \colon (G \times B,S) \rightarrow (Q,T)$ is a relational morphism, then the subgroup $N(\varphi)$ defined above is the 
smallest normal subgroup  $N$ of $G$ such that there is a flow $F \colon Q \rightarrow Rh_{G_{N}}(B_{N})$ given via the relational morphism that 
is the composition of the inverse of the morphism $(G \times B,S) \rightarrow (G_{N} \times B_{N}, \GM_{J}(G,N))$ with $\varphi$.

Let $N_1$ and $N_2$ be normal subgroups of $G_{J}$ such that each of $\GM_{J}(G_{J}/N_{i}), i=1,2$ has a flow with respect to a semigroup of 
complexity at most $k-1$. Then it is straightforward to check that $\GM_{J}(G_{J}/N_{1}\cap N_{2})$ also has such a flow. It follows that if for every 
regular $\mathcal{J}$-class $J$ of $S$, we define $N_J^{(k)}$ to be the smallest normal subgroup $N$ of $G$ such that 
$\GM_{J}(G/N)$ has a flow with respect to a  semigroup of complexity at most $k-1$ that this defines a normal subgroup spread. We call this the  
the $k$-th flow normal subgroup spread. By convention, $N_J^{(0)} = G_J$ is the maximal spread, that is, we take the whole group as the normal 
subgroup. More formally we make the following definition.

\begin{definition}
\label{flownormal.def}
    The $k$-th flow normal subgroup spread for a semigroup $S$ is the subgroup spread $N^{(k)}$ whose normal subgroup for regular $\mathcal{J}$-class 
    $J$ is $N_J^{(k)}$.
\end{definition}

Therefore, $N_{J}^{(k)}$ is the intersection of all subgroups of the form 
$N(\varphi)$, where $\varphi \colon (G \times B,S) \rightarrow (Q,T)$ for some
semigroup $T$ with $Tc \leqslant k-1$. As remarked above, $N_{J}^{(k)}$ is the smallest normal subgroup $N$ of $G_J$ such that $\GM_{J}(G,N)$ 
has a flow of complexity at most $k-1$. Thus $(G \times B, S) $ has a flow of complexity at most $k-1$ if and only if $N_{J}^{(k)}$ is the trivial subgroup. 

Although there is an infinite collection of relational morphisms of the form $\varphi \colon (G \times B,S) \rightarrow (Q,T)$ for some
semigroup $T$ with $Tc \leqslant k-1$, the following standard compactness 
argument shows that there exists a fixed relational morphism $\varphi \colon (G \times B,S) \rightarrow (Q_{k},T_{k})$ for some semigroup $T_{k}$ 
with complexity at most $k-1$ and such that $N(\varphi)=N_{J}^{(k)}$ and thus defines a flow for $\GM_{J}(G,N_{J}^{(k)})$ of complexity at most $k-1$.

\begin{thm}
Let $k>0$ and let $(G \times B,S)$ be a $\GM$ transformation semigroup. Then there exists a transformation semigroup $(Q_{k},T_{k})$ with 
$T_{k}c \leqslant k-1$ and a relational morphism $\varphi_{k}\colon (G \times B,S)\rightarrow (Q_{k},T_{k})$ such that $N(\varphi_{k})=N_{J}^{(k)}$.

\end{thm}

\begin{proof}
Define an equivalence relation on the collection of all relational morphisms $\varphi \colon (G \times B,S) \rightarrow (Q,T)$, where $T$ is a 
semigroup of complexity at most $k-1$ by setting $\varphi \colon (G \times B,S) \rightarrow (Q,T)$ equivalent to 
$\vartheta \colon (G \times B,S) \rightarrow (Q',T')$ if $N(\varphi)=N(\vartheta)$. Since $G$ is finite, there are a finite number $n$ of equivalence 
classes. Pick representatives of each equivalence class $\vartheta_{i} \colon (G \times B,S) \rightarrow (Q_{i},T_{i})$ for $i=1, \ldots, n$. The product morphism 
$\varphi_{k}\colon (G \times B,S) \rightarrow \prod_{i=1}^{n}(Q_{i},T_{i})$ satisfies $\varphi_{k}=\bigcap_{i=1}^{n}N(\vartheta_{i})$ and this is clearly $N_{J}^{(k)}$.

\end{proof}

We recall the notion of a pointlike subset of a finite semigroup $S$ with respect to a pseudovariety $V$. By definition a subset $X$ is pointlike with 
respect to $V$ if and only if for all relational morphisms $r \colon S \rightarrow T$ where $T \in V$, we have $X \subseteq tr^{-1}$ for some 
$t \in T$. The pointlike subsemigroup $\Pl_{V}(S)$ is the subsemigroup of the power subsemigroup $P(S)$ consisting of all subsets pointlike with 
respect to $V$. For more details see~\cite[Chapter 4]{RS.2009},

We now prove that $N_{J}^{(k)}$ is the maximal subgroup of $G_{J}$ that is pointlike for $C_{k}$ with $k>0$. The following is well-known and we include it 
for completeness.

\begin{lem}
\label{PLcontained}
Let $V$ and $W$ be pseudovarieties with $V \subseteq W$. Then $\Pl_{W}(S) \subseteq \Pl_{V}(S)$ for any semigroup $S$.
\end{lem}

\begin{proof}
Let $ X \in \Pl_{W}(S)$ and let $r \colon S \rightarrow T$ be a relational morphism with $T \in V$. Since $V \subseteq W$ it follows that there is a $t \in T$ 
such that $X \subseteq tr^{-1}$. Therefore $X \in \Pl_{V}(S)$.   
\end{proof}

The next lemma will allow us to reduce the computation of pointlike subgroups in $C_{k}$ to pointlike subgroups of $Gp*C_{k-1}$, where $Gp$ 
is the pseudovariety of finite groups. 

\begin{lem}
\label{PLreduction}
Let $V$ be a pseudovariety and $N$ a subgroup of a semigroup $S$. Then $N \in \Pl_{Ap*Gp*V}(S)$ if and only if $N \in \Pl_{Gp*V}(S)$. 
\end{lem}

\begin{proof}
Since $Gp*V \subseteq Ap*Gp*V$, Lemma \ref{PLcontained} implies that $\Pl_{Ap*Gp*V}(S) \subseteq \Pl_{Gp*V}(S)$. For the opposite inclusion let 
$N$ be a subgroup of $S$ and $N \in \Pl_{Gp*V}(S)$. Let $r \colon S \rightarrow T$ be a relational morphism with $T \in Ap*Gp*V$. There exists an 
aperiodic relational morphism $\theta \colon T \rightarrow U$ for some semigroup $U \in Gp*V$. 

Since $N \in \Pl_{Gp*V}(S)$ is a subgroup, we can choose an idempotent $e$ 
such that $N \subseteq e\theta^{-1}r^{-1}$. Note that $e\theta^{-1}$ is an 
aperiodic subsemigroup of $T$ since $\theta$ is an aperiodic morphism.
Therefore there is an idempotent $f \in e\theta^{-1}$ such that 
$N\subseteq fr^{-1}$. This proves that $N \in \Pl_{Ap*Gp*V}(S)$.
\end{proof}

We can now prove the main theorem of this subsection. First note that if $G$ is a maximal subgroup of a semigroup $S$, then there is a unique 
maximal subgroup of $S$ that is pointlike with respect to any pseudovariety $V$. This follows since a pointlike subgroup is contained in the inverse 
of an idempotent and thus its normal closure is also contained in the inverse image of that idempotent and that the product of pointlike sets is pointlike. 
Therefore the product of all normal subgroups of $G$ that are pointlike with respect to $V$ is the unique maximal pointlike subgroup of $G$ with 
respect to $V$. We use this remark  in the proof of the next theorem.

\begin{thm}
\label{maxpl}
Let $k>0$ and let $S$ be a $\GM$ semigroup with distinguished $\mathcal{J}$-class $J$ and $I(S) \approx M^{0}(A,G,B,C)$. 
Then $N_{J}^{(k)}$ is the maximal subgroup of $G$ that is pointlike with respect to $C_{k}$.    
\end{thm}

\begin{proof}
By Lemma~\ref{PLreduction} to show that $N_{J}^{(k)} \in \Pl_{C_{k}}(S)$ it is enough to prove that $N_{J}^{(k)} \in \Pl_{Gp*C_{k-1}}(S)$. By definition, 
$N_{J}^{(k)}$ is contained in a partition class of every admissible partition on the derived transformation semigroup 
$D(\varphi)$ of a relational morphism $\varphi$ from $S$ to a semigroup in $C_{k-1}$. It follows immediately from~\cite[Theorem 4.14.20]{RS.2009} 
that $N_J^{(k)} \in \Pl_{Gp*C_{k-1}}(S)$.

Now let $N$ be any subgroup of $G_{J}$ that is pointlike with respect to $Gp*C_{k-1}$. Again by~\cite[Theorem 4.14.20]{RS.2009}, $N$ is contained in 
a partition class of every admissible partition on the derived transformation semigroup $D(\varphi)$ of a relational morphism $\varphi$ from $S$ to a 
semigroup in $C_{k-1}$. Since $N_{J}^{(k)}$ is the intersection of all such subgroups, we thus have $N \subseteq N_{J}^{(k)}$.
\end{proof}

\begin{cor}
$N_{J}^{(k)}$ is both the maximal subgroup of $G_J$ that is pointlike with respect to $C_k$ and the minimal normal subgroup $N$ of $G_{J}$ 
such that $\GM_{J}(G_{J}/N)$ admits a flow of complexity at most $k-1$.
\end{cor}

\begin{remark}
\mbox{} 
\begin{enumerate}
\item  We will prove that membership in $N_J^{(k)}$ is decidable in Proposition \ref{njkdecidable}.
\item The slice theorem by B. Steinberg (see~\cite[Section 2.7]{RS.2009}) can be reformulated in terms of transformation semigroup as explained 
in the section on presentations~\cite[Section 4.14.2]{RS.2009}. This formulation proves that all $N_J^{(k)}$ give enough members of 
$\Pl_{C_{k}}(E_{k}^{+})$ to consider in the definition of $E_{k}^{+}$ and $E_{k}^{+}$ in Section~\ref{section.Mk}. We explain this in detail in 
the next subsection.
\end{enumerate}
\end{remark}

\subsection{Flow monoids and evaluation transformation semigroups }
\label{section.Mk}

In this subsection, we modify the definition of the $k^{th}$ flow monoid $M_{k}(L)$ in~\cite[Definition~4.26]{HRS.2012}, a submonoid of 
$\mathcal{C}(L^{2})$. This gives an improved version of the evaluation transformation semigroups of~\cite[Section 4]{HRS.2012}. These are used in 
our criterion for decidability of complexity.

The reader should look at the definition of $k$-loopable elements \cite[Definiton 4.12]{HRS.2012} and of the $k^{th}$ flow monoid 
$M_{k}(L)$ \cite[Definition 4.26]{HRS.2012}. We define a monoid $F_{k}(L)$ by keeping the first 4 parts of the definition of $M_{k}(L)$. We replace 
part (5) of the definition of $M_{k}(L)$  with the following new part (5F):

\begin{enumerate}
\item[(5F)] Let $k \geqslant 0$. If $f \in F_{k}(L)$ and $f \in N_J^{(k)}$ of the $k^{th}$-flow spread for $F_{k}(L)$, then 
$(f)^{\omega+\star} \in F_{k}(L)$.
\end{enumerate}

Equivalently by Theorem \ref{maxpl} we can formulate part (5F) as follows.

\begin{enumerate}
\item[(5F)] Let $k \geqslant 0$. If $f \in F_{k}(L)$ and $f$ belongs to the maximal subgroup of $G_{J}$ that is pointlike with respect to $C_{k}$ in $F_{k}(L)$, then 
$(f)^{\omega+\star} \in F_{k}(L)$.
\end{enumerate}

We now recall the evaluation transformation semigroup as defined in~\cite[Section 4]{HRS.2012}. It is based on an action of $M_{k}(L)$. To distinguish 
it from the modified evaluation transformation semigroup we define in this paper, we denote it by $\operatorname{Loopeval}_{k}(L)$. The evaluation 
transformation semigroup we use in this paper is based on an action of $F_{k}(L)$ and will be denoted by $\operatorname{Eval}_{k}(L)$. 

The vacuum $V$ is the join of $\overleftarrow{f}$ over all $f\in M_{k}(L)$. Note that $V$ is an idempotent operator and belongs to $M_{k}(L)$. 
For details about the Vacuum operator $V$ consult \cite[Section 4]{HRS.2012} especially Definitions~4.26 and~4.39. 

In the following we will be working in the submonoid $VM_{k}(L)V$, where $V$ is the Vacuum. It is very important to note the following. 
Let  $l \in LV$ and $f \in VM_{k}(L)V $. Then for all $l' \in L$ $(l,l')f=(l,l")$ for some $l" \in LV$. It follows that this gives an action of $VM_{k}(L)V$ on 
$LV$ defined by $lf = \operatorname{min}\{l" \mid (l,l") \in LVf\}$. See \cite[Section 4.5]{HRS.2012} for details. In particular, let $f \in M_{k}(L)$ and 
$f \in N_J^{(k)}$ the $k^{th}$-flow spread for $M_{k}(L)$. Then for all $l \in LV$, $lf^{\omega +*}$ is the least $l'$ such that $lf^{\omega}=l'$ and 
$l'f=l'$. We take this as the definition of $f^{\omega+*}$ in this paper.
In~\cite[Section 4]{HRS.2012} the Vacuum for $M_{k}(L)$ is denoted by $\mathcal{F}_{k}$.

We thus have a transformation semigroup with state set $LV$ and semigroup the image of $M_{k}(L)$ in the action above. We will restrict this 
action to a subset $\operatorname{States}_{k}(L)$ of $LV$. We refer the reader to \cite[Definition 5.1-5.2]{HRS.2012} for the definitions of Well-Formed Formulae
(WFFs) and the Standard Interpretation. We modify the definition of the set $\operatorname{States}_{k}(L)$  of \cite[Definition 5.3]{HRS.2012} by removing 
paragraph~(3) therein.
That is, we do not close the states under the partial order on $L$. This with our definition of $F_{k}(L)$ gives the transformation semigroup, 
called the Evaluation transformation semigroup $\operatorname{Eval}_{k}(L)=(\operatorname{States}_{k}(L),E_{k}(L))$, where $E_{k}(L)$ is the image 
of $VF_{k}(L)V$ in its action on $\operatorname{States}_{k}(L)$, an invariant subset of $LV$.

Before continuing, we compare the evaluation transformation monoid defined via $M_{k}(L)$ and our modified version based on our $F_{k}(L)$. Recall that in any semigroup $S$ and $s \in S$, we define 
$s^{\omega + 1}=ss^{\omega}=s^{\omega}s$. We note that $s^{\omega +1}$ is in the unique maximal subgroup of the subsemigroup generated by $s$.

\begin{lem}
Let $f \in M_{k}(L)$ or $f \in F_{k}(L)$. Then 
$f^{\omega +\star}=(f^{\omega+1})^{\omega+\star}$.
\end{lem}

\begin{proof}
Note that $(f^{\omega+1})^{\omega}=f^{\omega}$. Assume that for a state $l$, we have $lf^{\omega +\star}=l'$. Then $lf^{\omega}=l'$ and $l'f=l'$. It follows 
that $l'f^{n}=l'$ for all $n>0$ and in particular, $l'f^{\omega+1}=l'$. Conversely if $lf^{\omega}=l'$ and $l'f^{\omega+1}=l'$, it follows that $l'f = l'$. Therefore,
\[
	l(f^{\omega+1})^{\omega+\star}=(f^{\omega+1})^{\omega}(f^{\omega+1})^{\star}=
	lf^{\omega}(f^{\omega+1})^{\star}=lf^{\omega}f^{\star}=lf^{\omega+*}.
\]
\end{proof}

Since $f^{\omega+1}$ is a group element, the following corollary is immediate.

\begin{cor}
\label{gpelt}
In defining the transformation semigroups $\operatorname{Loopeval}_{k}(L)$ and $\operatorname{Eval}_{k}(L)$ we need only apply the 
${\omega+\star}$ operator to elements that belong to subgroups of $M_{k}(L)$ and $F_{k}(L)$ respectively.    
\end{cor}

\begin{prop}
\label{loopsubeval}
For every $k \geqslant 0$, $\operatorname{Loopeval}_{k}(L)$ is a sub-transformation semigroup of $\operatorname{Eval}_{k}(L)$. 
\end{prop}

\begin{proof}
It follows from Proposition \ref{gpelt} and~\cite[Proposition 4.13]{HRS.2012} that any group element of $M_{k}(L)$ is pointlike for $C_{k}$ in its image as 
an element of the action semigroup of $\operatorname{Loopeval}(L)$. Therefore it belongs to $N_{J}^{(k)}(L)$ by Theorem~\ref{maxpl} and the result follows.    
\end{proof}

From hereon in, we only use $F_{k}(L)$ and $\operatorname{Eval}_{k}(L)$. We use the definition of $\operatorname{States}_{k}(L)$ 
from~\cite[Section 4]{HRS.2012} in this context. For brevity, we write $\operatorname{States}^{+}_{k}$ for $\operatorname{States}_{k}(\SP(G\times B))$, 
$E^{+}_{k}$ for $E_{k}(\SP(G\times B))$, $\operatorname{States}^{-}_{k}$ for $\operatorname{States}_{k}(Rh_{B}(G))$ and $E^{-}_{k}$ for 
$E_{k}(Rh_{B}(G))$. We give the connections between these sets. We will identify $Rh_{B}(G)$ 
with $\CS(G\times B)$ as in Proposition \ref{Updown}.

Fix a $b_{0} \in B$. We call $\sigma=((1,b_{0}), \{(1,b_{0})\})$ the start state of $\operatorname{States}^{+}_{k}$. We note that 
$\operatorname{States}^{+}_{k} = (1,b_{0})E^{+}_{k}$. We call $G\sigma$ the start state for $\operatorname{States}^{-}_{k}$.  
Note that $G\sigma$ has set $G \times \{b_{0}\}$ with the singleton partition. Furthermore, $G\sigma$ is the 
state corresponding to the unique element of $Rh_{B}(G)$ with set $\{b_{0}\}$ in the isomorphism from Proposition \ref{Updown}. We thus have that 
$\operatorname{States}^{-}_{k}=(G\sigma) E_k^{-}$. 
In the next lemma we use the notation $L{\downarrow}$ for the set of elements less than or equal to some element in $L$.

\begin{lem}\label{States.lem}
$\operatorname{States}_{k}^{+}\downarrow$  $=$  $\operatorname{States}_{k}^{-}\downarrow$.
\end{lem}

\begin{proof}
Since $\sigma \leqslant G\sigma$ it follows immediately by induction on WFFs and the definition of $\operatorname{States}$ that every element of 
$\operatorname{States}^{+}_{k}$ is less than or equal to some element of $\operatorname{States}^{-}_{k}$. Therefore 
$\operatorname{States}_{k}^{+}\downarrow \subseteq \operatorname{States}_{k}^{-}\downarrow$.

For the opposite containment we will show that $\operatorname{States}_{k}^{-} \subseteq \operatorname{States}_{k}^{+}$. To do this it is enough to show 
that $G\sigma \in \operatorname{States}_{k}^{+}$.

Let $g \in G$. Then $\sigma g^{\omega+*} =\langle g \rangle \sigma \in \operatorname{States}^{+}_{k}$, where $\langle g\rangle$ is the subgroup of 
$G$ generated by $g$. Every finite group $G$ is the product of its cyclic subgroups in some order and with repeats allowed. Indeed, let $G$ be generated 
by $\{g_{1}, \ldots, g_{n}\}$. The regular expression identity
$\{g_{1}, \ldots , g_{n}\}^{\star} = (g_{1}^{\star}\ldots g_{n}^{\star})^{\star}$ gives this result immediately.
It follows that $G\sigma \in \operatorname{States}^{+}_{k}$ as desired.
\end{proof}

The following corollary follows immediately from Lemma \ref{States.lem}.

\begin{cor} 
\label{prop.eval}
Let $k \geqslant 0$. Then some state of $\operatorname{Eval}_{k}^{+}$ is a contradiction if and only if some state of
$\operatorname{Eval}_{k}^{-}$ is a contradiction.
\end{cor}

The following is~\cite[Proposition 5.16]{HRS.2012}.

\begin{prop}
\label{prop.Mk}
Let $S$ be a $\operatorname{GM}$ semigroup. Then $S$ is a subsemigroup of $\operatorname{Eval}_{k}^{+}(\SP(G\times B))$ for all $k \geqslant 0$.
\end{prop}

We now claim that the lower bound of~\cite[Section 5]{HRS.2012} holds by replacing $\operatorname{Loopeval}_{k}(L)$ with $\operatorname{Eval}_{k}(L)$. 
This follows since Proposition \ref{loopsubeval} implies that if $\operatorname{Eval}_{k-1}(L)$ has no contradiction, then neither does 
$\operatorname{Loopeval}_{k-1}(L)$

\begin{thm}\label{lb}
(Lower bounds) If $\operatorname{RLM}(S)c=k$ and some state of $\operatorname{Eval}_{k-1}(L)$ is a contradiction, then $Sc=k+1$.
\end{thm}

\section{Decidability of complexity for all finite semigroups and automata}
\label{section.7}

We now formulate our condition of deciding whether a semigroup has complexity $k$, which states that the lower bound of Theorem \ref{lb} computes 
complexity exactly. Namely, let $k \geqslant 1$.  Let $S$ be a $\operatorname{GM}$ semigroup with $\operatorname{RLM}(S)c=k$. Then our condition 
says that $Sc = k$ if and only if no state of $\operatorname{Eval}_{k-1}^{-}$ is a contradiction. 
We remark by Proposition~\ref{prop.eval} this is equivalent to replacing $\operatorname{Eval}_{k-1}^{-}$ with $\operatorname{Eval}_{k-1}^{+}$ and 
demanding no contradictions for states of this transformation semigroup.

We formulate our main theorem as follows.

\begin{thm}
\label{mainthm}
Let $S$ be a $\GM$ semigroup with corresponding transformation semigroup $(G \times B,S)$. Then there is a transformation semigroup $(Q,T)$ with 
$Tc \leqslant k-1$ and a flow $Q \rightarrow Rh_{B}(G)$ if and only if $\operatorname{Eval}_{k-1}(Rh_{B}(G))$ has no contradiction. Consequently, 
$Sc\leqslant k$ if and only if $\RLM(S)c \leqslant k$ and $\operatorname{Eval}_{k-1}(Rh_{B}(G))$ has no contradiction.
\end{thm}

The reader should pay attention to the subscripts. The new lower bound for all $c=k$ is the minimal $k-1$, so that $\operatorname{Eval}_{k-1}$  
has no contradiction. 

The case of complexity 0 is trivial. The main theorem of~\cite{MRS.2023} shows that this condition is correct for the case $k=1$. We recall the result here.

\begin{thm}
\label{mainc=1}
Let $S$ be a $\GM$ semigroup with corresponding transformation semigroup $(G \times B,S)$. Then there is an aperiodic transformation semigroup 
$(Q,T)$ and a flow $Q \rightarrow Rh_{B}(G)$ if and only if $\operatorname{Eval}_{0}(Rh_{B}(G))$ has no contradiction. Consequently, $Sc = 1$ 
if and only if $\RLM(S)c \leqslant 1$ and $\operatorname{Eval}_{0}(Rh_{B}(G))$ has no contradiction.
\end{thm}

Note that $\operatorname{Loopeval_{0}(L)} = \operatorname{Eval}_{0}(L)$ since we are allowed to apply the $\omega+*$ operator to every element 
and thus $M_{0}(L)=F_{0}(L)$.

Unlike for the complexity $1$ case treated in~\cite{MRS.2023},
the bound using $\operatorname{Loopeval}_{k}$ of~\cite{HRS.2012} is probably not sharp for complexity greater than~$1$. For example, we conjecture 
that~\cite[Conjecture 1.9 and Definition 5.9]{Rhodes.1977} gives a semigroup $S$ with $Sc = 3$, 
but for which the lower bound of~\cite{HRS.2012} is $2$. 
See also~\cite{RS.2006}. This is because the notion of 
$k$-loopable in~\cite{HRS.2012} is probably too weak. We strengthened the notion of $k$-loopable 
in this paper precisely because of this problem. In a future paper we prove that the lower bound based on $k$-loopable elements is a proper lower bound.

The proof of the main theorem of this paper uses the decidability of complexity $1$ of~\cite{MRS.2023} and follows the proof scheme of~\cite{HRS.2012} with our improved lower bound.

\subsection{Proof of decidability}
\label{section.proof}

In this subsection we prove Theorem~\ref{mainthm}. We first note that our improved lower bound, Theorem~\ref{lb} is precisely the necessary 
direction of our main theorem. We record this here.

\begin{thm}
\label{necmainthm}
Let $S$ be a $\GM$ semigroup with corresponding transformation semigroup $(G \times B,S)$. If there is a transformation semigroup $(Q,T)$ 
with $Tc \leqslant k-1$ and a flow $Q \rightarrow Rh_{B}(G)$, then $\operatorname{Eval}_{k-1}(Rh_{B}(G))$ has no contradiction. Consequently 
if $\RLM(S)c=k$, then $Sc=k$.
\end{thm}

We begin the proof in the sufficient direction. By induction we assume that the sufficient condition in Theorem \ref{mainthm} has been proved for $k-1$, where $k>0$. We note that Theorem \ref{mainc=1} is the case $k=1$. 

\begin{prop}
\label{njkdecidable}
Let $S$ be a $\GM$ semigroup. Then for all $k\geqslant 0$ membership in the subgroup 
$N_{J}^{(k)}$ where $J$ is a $\mathcal{J}$-class of the semigroup of $\operatorname{Eval}_{k}(L)$ is decidable. 
\end{prop}

\begin{proof}
The case $k=0$ is trivial, since by definition $N_{J}^{(0)}$ is the maximal subgroup $G_J$. Assume that the assertion is true for $k-1$. 

Let $J$ be a $\mathcal{J}$-class of $\mathcal{C}(L^{2})$ and let $N$ be a normal subgroup of $G_{J}$. By the induction hypothesis for the main theorem 
we have that $\GM_{J}(G_{J}/N)$ has a flow of complexity at most $k-1$ if and only if $\operatorname{Eval}_{k-1}(L)$ relative to the semigroup 
$\GM_{J}(G_{J}/N)$ has no contradiction. This is decidable by the inductive hypothesis. We therefore can effectively compute the minimal such 
normal subgroup which by definition is $N_{J}^{(k)}$.    
\end{proof}

As the definition of $\operatorname{Eval}_{k}(L)$ involves computing all subgroups of the form $N_{J}^{(k)}(L)$, by induction on $k$ we have the 
following corollary.

\begin{cor}
Let $k\geqslant 0$. Then $F_{k}(L)$ is a computable submonoid of $\mathcal{C}(L^{2})$. Therefore $\operatorname{Eval}_{k}(L)$ is a computable 
transformation semigroup for all $k \geqslant 0$.
\end{cor}

Recall that for $k>0$ the sufficient condition is that $\operatorname{Eval}_{k-1}(L)$ has no contradiction. We need to construct a transformation 
semigroup $(Q,T)$ with $Tc \leqslant k-1$ and a flow $f \colon Q \rightarrow Rh_{B}(G)$. More precisely, let $S$ be a $\GM$ semigroup. 
We assume that that $\operatorname{Eval}_{k-1}^{+}(\operatorname{SP}(G \times B))$ has no contradiction.

By Lemma~\ref{States.lem}, we can replace the latter assumption by $\operatorname{Eval}_{k-1}^{-}(\operatorname{Rh}_{B}(G))$ has no contradiction.  
For convenience we denote this transformation semigroup by $\operatorname{Eval}_{k-1}(L)$.

We now construct a relational morphism of transformation semigroups $r\colon (\operatorname{Eval}_{k-1}(L),Z) \to (W_{k-1},Z)$, where 
$Z=\operatorname{Eval}_{k-1}(L)$ is taken as generating set. It satisfies the following properties:
\begin{enumerate}
 \item{$W_{k-1}c \leqslant k-1$}
 \item{For each regular $\mathcal{J}$-class $J$ of $\operatorname{Eval}^{+}_{k-1}$, and for each idempotent $e \in W_{k-1}$, 
 $er^{-1} \cap G_{J} = N_{J}^{(k-1)}$.}
\end{enumerate}

This relational morphism is computable. This is because by Proposition~\ref{njkdecidable} membership in $N_{J}^{(k-1)}$ is decidable and by 
Theorem~\ref{maxpl} is the maximal subgroup of $G_{J}$ that is pointlike with respect to $C_{k-1}$. Therefore we can effectively enumerate all relational 
morphisms from $\operatorname{Eval}_{k-1}(L)$ to transformation semigroups of complexity at most $k-1$ and we are guaranteed effectively that at 
least one will have the property that $er^{-1} \cap G_{J} = N_{J}^{(k-1)}$ for all idempotents $e$.

We now apply the geometric semigroup expansion $\operatorname{GST}$ as defined in~\cite[Section 4]{MRS.2023}
\begin{equation}
\label{equation.construction}
	\varphi \colon \operatorname{GST}(\#r,Z)  \stackrel{\varphi}{\twoheadrightarrow} \operatorname{GST}(W_{k-1}(Rh_B),Z),
\end{equation}
where $\#r$ is the graph of $r$.

 In other words, the main construction is a subsemigroup of 
$\operatorname{Eval}_{k-1}^{+}(\operatorname{SP}(G \times B)) \times W_{k-1}(\operatorname{SP}(G \times B))$
generated by $(z,z)$ for $z\in Z$. 

\begin{lem}
\label{lemma.main}
Let $\varphi$ be the relational morphism defined above. Let $D(\varphi)$ be its derived transformation semigroup. 
Then $D(\varphi)c \leqslant 1$.
\end{lem}

\begin{proof}
We want to show that $D(\varphi)c \leqslant 1$, so we apply the main results of~\cite{MRS.2023}. 

We place an order $>$ on the $\GM$ images of any finite semigroup semigroup:
\[
	\operatorname{GM}_j > \operatorname{GM}_{j+1} \qquad \text{if and only if} \qquad
	\operatorname{GM}_j \twoheadrightarrow \operatorname{RLM}(\operatorname{GM}_j) \twoheadrightarrow \operatorname{GM}_{j+1}.
\]
The complexity $c$ of any semigroup is the longest chain of pure $\operatorname{GM}$ in this order, where a pure $\GM$ semigroup is one whose 
complexity is one more than its $\RLM$ image. See~\cite[Chapter 5]{Arbib.1968} where the result is proved for completely regular semigroups and 
the proof  generalizes to arbitrary finite semigroups.

To prove that $D(\varphi)$ has complexity at most 1, we induct on the $\GM$ image of $D (\varphi)$ by a fixed linear extension of the order
$>$ defined above. We  denote this by $\tilde{>}$. The minimum element of $\tilde{>}$ is a $0$-simple non-aperiodic $\GM$ semigroup. Therefore
$(\operatorname{GM}_{\min}) \operatorname{RLM} c = 0$ and hence $(\operatorname{GM}_{\min}) c = 1$.

Consider the first case $\operatorname{GM}_j$ which has not been considered yet in $\tilde{>}$ order, we have
\[
	\operatorname{GM}_j \twoheadrightarrow \operatorname{RLM}(\operatorname{GM}_j)  \twoheadrightarrow  \operatorname{GM}_x,
\]
where $\operatorname{GM}_x$ is any $\operatorname{GM}$ image of $ \operatorname{RLM}(\operatorname{GM}_j) $. Since all of these 
are strictly smaller than $\operatorname{GM}_j$, by induction we have $(\operatorname{GM}_j) \operatorname{RLM} c\leqslant 1$, so we can
apply~\cite{MRS.2023} to $\operatorname{GM}_j$. By~\cite{MRS.2023} we need to check that $M_0(\operatorname{GM}_j) = 
\operatorname{Eval}(\operatorname{GM}_j)$ has no contradiction.

We are working in the $\mathcal{J}$-class $J_{j}$ in the semigroup of $D(\varphi)$. 

We must show that the criterion for complexity 1 from~\cite{MRS.2023} 
holds for $\operatorname{GM}(J_j)$, that is, that $\operatorname{Eval}_{k-1}^{+}$ has no contradictions. 

Let $Z'$ be generators for $\operatorname{GM}(J_j)$. Pull them back to the semigroup of $D(\varphi)$.

Call the resulting set $Z$. By the derived category theory $Z$ determines a member $Z$ of $E_{k-1}^{+}$ and a subset of states of 
$\operatorname{Eval}_{k-1}^{+}$ denoted by $\operatorname{Sub}(Z)$. We choose an arbitrary Well-Formed Formula $W$ for 
$\operatorname{Eval}_{k-1}^{+}(\operatorname{GM}(J_{j}))$ with variables $Z_{1}', \ldots, Z_{N}'$. We now replace $W(Z_{1}', \ldots , Z_{N}')$ 
with $W(Z_{1}, \ldots , Z_{N})$. By assumption $W(Z_{1}, \ldots , Z_{N})$ has no contradictions. Therefore  $W(Z_{1}', \ldots , Z_{N}')$ has no contradictions. 
The main result of \cite{MRS.2023} now implies that $D(\varphi)$ has complexity at most 1. 
\end{proof}

In conclusion, we summarize the proof of the theorem. Let $S$ be a $\GM$ semigroup and $k > 0$. By induction on cardinality we can 
decide if $\RLM(S)c \leqslant k$. If the answer is no, then $Sc>k$. If by induction on $k$ we have $\RLM(S)c < k$, then $Sc \leqslant k$. So we can 
assume that $\RLM(S)c = k$. We assume that $\operatorname{Eval}_{k-1}(L)$ has no contradiction. Let 
$\varphi \colon \operatorname{GST}(\operatorname{Eval}_{k-1}) \rightarrow \operatorname{GST}(W_{k-1})$ be the relational 
morphism we have constructed above. The main part of the proof shows that the derived transformation semigroup $D(\varphi)$ has complexity 
at most 1. More precisely, having no contradictions in $\operatorname{Eval}_{k-1}$ implies 
that there are no contradictions in $\operatorname{Eval}_{0}$ of any $\GM$ image of $D(\varphi)$.

\bibliographystyle{plain}
\bibliography{complexity-n}{}

{(Stuart Margolis) Department of Mathematics, Bar Ilan University, Ramat Gan 52900, Israel

{\it Email address}: \; \texttt{margolis@math.biu.ac.il}}

\medskip

{(John Rhodes) Department of Mathematics, University of California, Berkeley, CA 94720, U.S.A.

{\it Email address}: \; \texttt{rhodes@math.berkeley.edu, blvdbastille@gmail.com}}

\medskip

{\large (Anne Schilling) Department of Mathematics, UC Davis, 
One Shields Ave., Davis, CA 95616-8633

{\it Email address}:\;\texttt{anne@math.ucdavis.edu}}

\end{document}